\documentclass[12pt]{extarticle}
\usepackage{amsmath, amsthm, amssymb, color}
\usepackage[colorlinks=true,linkcolor=blue,urlcolor=blue]{hyperref}
\usepackage{graphicx}
\usepackage{caption}
\usepackage{mathtools}
\usepackage{enumerate}
\usepackage{verbatim}
\usepackage{nicematrix}
\usepackage{tikz,tikz-cd,tikz-3dplot}
\usepackage{amssymb}
\usetikzlibrary{matrix}
\usetikzlibrary{arrows}
\usepackage{algorithm}
\usepackage{caption}
\usepackage{subcaption}
\oddsidemargin \evensidemargin
\usepackage{makecell}
\usepackage{array}

\newtheorem{theorem}{Theorem}
\newtheorem{proposition}[theorem]{Proposition}
\newtheorem{lemma}[theorem]{Lemma}
\newtheorem{corollary}[theorem]{Corollary}

\theoremstyle{definition}

\newtheorem{remark}[theorem]{Remark}

\newtheorem{example}[theorem]{Example}

\newtheorem{question}[theorem]{Question}
\numberwithin{theorem}{section}

\newcommand{\PP}{\mathbb{P}}
\newcommand{\RR}{\mathbb{R}}

\newcommand{\CC}{\mathbb{C}}
\newcommand{\ZZ}{\mathbb{Z}}

\title{\bf Minimal Kinematics on $\mathcal{M}_{0,n}$}

\author{Nick Early, Ana\"elle Pfister and Bernd Sturmfels}

\date{}
\begin{document}
\maketitle

\begin{abstract}
  \noindent
  Minimal kinematics identifies likelihood degenerations 
  where the critical points are given by rational formulas.
  These rest on the Horn uniformization of Kapranov-Huh.
We characterize all choices of minimal kinematics on the moduli space $\mathcal{M}_{0,n}$.
These choices are motivated by the CHY model in physics and they are
represented combinatorially by 2-trees. We  compute 2-tree amplitudes, and we 
explore extensions to non-planar on-shell diagrams, here identified
with the hypertrees of Castravet-Tevelev.
  \end{abstract}

\section{Introduction}

The moduli space $\mathcal{M}_{0,n}$ of $n$ labeled points
on the projective line $\PP^1$ plays a prominent role in algebraic geometry
and its interactions with combinatorics. It is equally important in
 physics where it is used, for example, in the CHY model \cite{CHY} to 
compute scattering amplitudes.
The space $\mathcal{M}_{0,n}$ is a very affine variety
of dimension $n-3$, with coordinates given by the  matrix
\begin{equation}
\label{eq:M0nmatrix}
X \,\,\, = \,\,\, \begin{bmatrix}
\,1 & 1 & 1 & 1 & \cdots & 1 & 0 \, \,\\
\,0 & 1 & x_1 & x_2 & \cdots & x_{n-3} & 1 \,\,
\end{bmatrix}.
\end{equation}
More abstractly, $\mathcal{M}_{0,n}$ is the quotient of the
open Grassmannian ${\rm Gr}(2,n)^o$ by the action of the 
torus $(\CC^*)^n$.  We write $p_{ij}$ for the Pl\"ucker coordinates
on ${\rm Gr}(2,n)$. These are the $2 \times 2$ subdeterminants of $X$,
and being in ${\rm Gr}(2,n)^o$ means that $p_{ij} \not= 0$ for all $1 \leq i < j \leq n$.

A basic ingredient in the CHY model is the following {\em scattering potential} on 
 $\mathcal{M}_{0,n}$:
\begin{equation}
\label{eq:scattering}
 L \,\,\, = \,
\sum_{1 \leq i < j \leq n} s_{ij} \cdot {\rm log}( p_{ij}) .
\end{equation}
The coefficients $s_{ij}$ are known as {\em Mandelstam invariants}.
Using the conventions that  $\,s_{ii} = 0$ and $s_{ji} = s_{ij}$,
the Mandelstam invariants must satisfy the {\em momentum conservation relations}
\begin{equation}
\label{eq:momentumcons} \sum_{j=1}^n s_{ij} \,\,=\,\, 0 \quad \hbox{for all} \,\,\, i \in \{1,2,\ldots,n\}. 
\end{equation}
These relations ensure that $L$ is well-defined 
on $\mathcal{M}_{0,n} = {\rm Gr}(2,n)^o/(\CC^*)^n$, for any  branch of the logarithm function.
The following result on the critical points of $L$ is well-known.

\begin{proposition} For a general choice of $s_{ij}$, the scattering potential $L$ has $(n-3)!$ complex critical points on $\mathcal{M}_{0,n}$.
If the $s_{ij}$ are real numbers, then all $(n-3)!$ critical points are real.
\end{proposition}

We refer to \cite[Section 2]{ST} for a proof, computational aspects,
and a statistics perspective. 

\begin{example}[$n=4$]
Up to an additive constant $s_{23}  \cdot {\rm log}(-1)$, we have
\begin{equation}
\label{eq:tiny}
L \,\,=\,\, s_{13} \cdot {\rm log}(x_1) \,+ \,s_{23} \cdot {\rm log}(1-x_1). 
\end{equation}
This is the log-likelihood function for coin flips with bias $x_1$, when
 heads resp.~tails were observed $s_{13}$ resp.~$s_{23}$ times.
The unique critical point of $L$ is the {\em maximum likelihood~estimate}:
\begin{equation}
\label{eq:simpleformula} \hat x_1 \,\, = \,\, \frac{s_{13}}{s_{13} + s_{23}}. 
\end{equation}
\end{example}

For $n \geq 5$  there is no simple formula
because the {\em ML degree} is $(n-3)!$. 
Minimal kinematics \cite{CE,Early} provides
an attractive alternative.
The idea is to set a few $s_{ij}$ to zero until we 
reach a model of ML degree one \cite{ABFKST, DMS, Huh}.
This yields nice
rational functions,  like (\ref{eq:simpleformula}), for all~$n$.
 
\begin{example}[$n=6$] \label{ex:firstchoice}
One choice of minimal kinematics is the codimension $3$ subspace defined by
$s_{24} =  s_{25} = s_{35} = 0$. After this substitution,
the scattering potential (\ref{eq:scattering}) becomes
$$L_T = s_{13} \log \left(x_1\right)+s_{14} \log \left(x_2\right)+s_{15} \log \left(x_3\right)+s_{23} \log \left(x_1{-}1\right)+s_{34} \log \left(x_2{-}x_1\right)+s_{45} \log \left(x_3{-}x_2\right). $$
Our task is to solve the scattering equations  $\nabla L_T = 0$. Explicitly, these equations are
$$	\frac{s_{13}}{x_1}+\frac{s_{23}}{x_1-1}-\frac{s_{34}}{x_2-x_1} \,\, = \,\,
		\frac{s_{14}}{x_2}+\frac{s_{34}}{x_2-x_1}-\frac{s_{45}}{x_3-x_2} \,\, = \,\,
			\frac{s_{15}}{x_3}+\frac{s_{45}}{x_3-x_2}\, \,\, =\, \,\, 0. $$
This system is easily solved, starting from the end.  Back-substituting gives the critical~point:
$$ \hat x_1 \,=\, \frac{s_{13} + s_{14} + s_{15} + s_{34} + s_{45}}{s_{13} + s_{14} + s_{15} + s_{23} + s_{34} + s_{45}} 
\,, \,\,
\hat x_2 \,=\, \hat x_1 \cdot \frac{ s_{14} + s_{15} + s_{45}}{s_{14} + s_{15} + s_{34} +  s_{45}}
\, , \,\, \hat x_3 \,=\, \hat x_2 \cdot \frac{s_{15}}{s_{15} + s_{45}}.$$
\end{example}

The key observation is that numerators and denominators are
products of linear forms with positive coefficients. This is 
characteristic of all models of ML degree one, thanks
to a theorem of Huh \cite{Huh}. This goes back to Kapranov \cite{Kap}
who coined the term {\em Horn  uniformization}.
We seek to find all Horn uniformizations of the
      moduli space~$\mathcal{M}_{0,n}$. 
      Assuming the representation in (\ref{eq:M0nmatrix}),
      the solution is given by graphs called $2$-trees.
      Their~vertices are indexed by $[n-1] = \{1,2,\ldots,n-1\}$.
      The following is our first main result in this article.
      
\begin{theorem} \label{thm:vier} Choices of minimal kinematics on $\mathcal{M}_{0,n}$
are in bijection with 2-trees on~$[n{-}1]$.
\end{theorem}      

Here, we use the following formal definition of {\em minimal kinematics}.
Fix the set of index pairs $(i,j)$ in (\ref{eq:scattering}) whose
corresponding $ 2 \times 2$ minor in the matrix $X$
is non-constant. This is
\begin{equation}
\label{eq:setS} S \, = \, \bigl\{ (i,j) : 1 \leq i < j \leq n-1  \bigr\} \setminus \{(1,2) \}. 
\end{equation}
For any subset $T$ of $S$, we restrict  to the kinematic subspace where $s_{ij}=0$ for all $ij \in S\setminus T$:
\begin{equation}
\label{eq:scattering2}
 L_T \,\,\, = \,
\sum_{(i,j) \in T} s_{ij} \cdot {\rm log}( p_{ij}),
\end{equation}
We say that $T$ exhibits {\em minimal kinematics} for $\mathcal{M}_{0,n}$
 if  the function (\ref{eq:scattering2}) has exactly one
 critical point, which is hence rational in the  $s_{ij}$, and $T$ is inclusion-maximal with this property.
With this definition, the minimal kinematics for $\mathcal{M}_{0,6}$ in Example \ref{ex:firstchoice} is 
exhibited by the subset
\begin{equation}
\label{eq:firstchoice} T  \, = \, \bigl\{ (1,3), (2,3), \,(1,4), (3,4), \,(1,5), (4,5) \bigr\}. 
\end{equation}
The proof of Theorem \ref{thm:vier} is given in Section \ref{sec2}.
Section \ref{sec3} features the Horn uniformization. Formulas for
the unique critical point of $L_T$ are given in
Theorem \ref{thm:hornmain} and Corollary \ref{cor:hatp}.

In Section \ref{sec4} we introduce an {\em amplitude} $m_T$ for any 2-tree~$T$,
and we compute $m_T$ in Theorem \ref{thm:treeamp}.
The following example explains why we use the term amplitude. It is 
aimed at readers from physics who are familiar with
the {\em biadjoint scalar amplitude}~$m_n$; see \cite[Section~3]{CHYB}.
We recall that $m_n$ is the integral of the {\em Parke-Taylor factor}
$1/(p_{12} p_{23} \cdots p_{n-1,n} p_{n,1})$ over 
the moduli space $\mathcal{M}_{0,n}$, localized to the solutions to the scattering equations
$\nabla L = 0$.

\begin{example}[$n=6$] \label{ex:intro6}
The biadjoint scalar amplitude for 
$\mathcal{M}_{0,6}$ is the rational function
\begin{equation}
\label{eq:m6}
\!\! \begin{matrix}
m_6 & \! = \!\!\, & \frac{1}{s_{12} s_{34} s_{56}}
+\frac{1}{s_{12} s_{56}  s_{123}}
+\frac{1}{s_{23} s_{56}  s_{123}}
+\frac{1}{s_{23} s_{56}  s_{234}}
+\frac{1} {s_{34} s_{56} s_{234}}
+ \frac{1}{s_{16} s_{23} s_{45}}
+ \frac{1}{ s_{12} s_{34} s_{345}} \smallskip  \\ & 
\!\! & \,\, +\, \frac{1}{s_{12} s_{45} s_{123}}
+ \frac{1}{s_{12} s_{45} s_{345}}
+\frac{1}{s_{16} s_{23} s_{234}}
+\frac{1} {s_{16} s_{34} s_{234}}
{+} \frac{1}{s_{16} s_{34} s_{345}}
{+}\frac{1}{s_{16} s_{45} s_{345}}
{+} \frac{1}{s_{23} s_{45} s_{123}}.
\end{matrix}
\end{equation}
Here $s_{ijk} = s_{ij} + s_{ik} + s_{jk}$.  
In physics literature this amplitude is usually denoted $ m(\mathbb{I}_6,\mathbb{I}_6)$,
 where $\mathbb{I}_6 = (123456)$ is the standard cyclic order.
See also \cite[eqn (2.7)]{Cruz} or \cite[eqn (23)]{ST}.

The $14$ summands in (\ref{eq:m6})  correspond to the
vertices of the $3$-dimensional associahedron. The fomula is unique
since the nine planar kinematic invariants which appear form a basis of the dual kinematic space.  
With an eye towards more general situations, to achieve a unique formula
for $m_n$ modulo the relations (\ref{eq:momentumcons}),
we may also use the basis $S$ from (\ref{eq:setS}). Thus, we set
$$ \begin{small} \begin{matrix}
s_{12}\, \,=\,\, -s_{13}-s_{14}-s_{15}-s_{23}-s_{24}-s_{25}-s_{34}-s_{35}-s_{45},
\\ 
s_{16} \,=\, s_{23}+s_{24}+s_{25}+s_{34}+s_{35}+s_{45} ,\,\,
 s_{26} \,=\, s_{13}+s_{14}+s_{15}+s_{34}+s_{35}+s_{45},
 \\ 
  s_{36} \,=\, -s_{13}-s_{23}-s_{34}-s_{35},\,\,
  s_{46} \,=\,  -s_{14}-s_{24}-s_{34}-s_{45},\,\,
  s_{56} \,=\, -s_{15}-s_{25}-s_{35}-s_{45}.
\end{matrix} \end{small}
$$
We now restrict to 
the minimal kinematics in Example \ref{ex:firstchoice}.
For $s_{24} =  s_{25} = s_{35} = 0$, we find
$$ m_6\big\vert_{s_{24}=s_{25}=s_{35}=0} 
	\,\,=\,\,  
	 \frac{
		(s_{13}+s_{14}+s_{15}+s_{34}+s_{45})\,(s_{14}+s_{15}+s_{45})\, s_{15} }{
		s_{23} \,(s_{13}{+}s_{14}{+}s_{15}{+}s_{23}{+}s_{34}{+}s_{45}) \, s_{34}\,   
		(s_{14}{+}s_{15}{+}s_{34}{+}s_{45}) \,		s_{45} \,(s_{15}{+}s_{45})  }. 
$$
Up to relabeling, this is the amplitude $m_{T_1}$
we associate with  the 2-tree $T_1$ in (\ref{eq:twotrees}). 
This illustrates Theorem \ref{thm:treeamp}.
The factors are explained by the Horn matrix
$H_{T_1}$ in Example~\ref{ex:n=6}.
We note that,
modulo momentum conservation (\ref{eq:momentumcons}),
our amplitude  simplifies to
$$
	m_6\big\vert_{s_{24}=s_{25}=s_{35}=0} \,\,=\,\, \frac{\left(s_{12}+s_{23}\right) \left(s_{45}+s_{56}\right) \left(s_{34}+s_{456}\right)}{s_{12} s_{23} s_{34} s_{45} s_{56} s_{456}}.
	$$
\end{example}

In Section \ref{sec5} we venture into territory
that is of great significance for both
algebraic geometry  and particle physics.
Structures we attache to 2-trees generalize naturally to
the hypertrees of Castravet and Tevelev  \cite{CT}.
By \cite[Lemma 9.5.(3)]{Tev}, hypertrees are
equivalent to the on-shell diagrams of 
Arkani-Hamed, Bourjaily, Cachazo, Postnikov and Trnka  \cite{beyond}.
Every hypertree $T$ has an associated amplitude $m_T$.
This rational function does not admit a Horn formula,
like that in (\ref{eq:huhformula}) for 2-trees $T$. Indeed, now the ML degree 
for $T$ is larger than one. Here, we aim to reach minimal kinematics by restricting to
subspaces of kinematic space. This is seen in 
Example \ref{ex:yes}. The paper concludes with questions for future research.

\section{2-trees}
\label{sec2}

We begin by defining the class of graphs referred to in Theorem \ref{thm:vier}.
A {\em 2-tree} is a graph $T$ with $2n-5$ edges on the
 vertex set $[n-1] = \{1,2,\ldots,n-1\}$,
which can be constructed inductively as follows.
We start with the single edge graph $\{12\}$, which is the unique
2-tree for  $n-1=2$. For $k=3,4,5,\ldots,n-1$, we proceed inductively as follows:
we select an edge $ij$ whose vertices $i$ and $j$ are in $\{1,\ldots,k-1\}$,
and we introduce the two new edges $ik$ and $jk$.
There are $2k-5$ choices at this stage, so our process leads to
$\,(2n-7)!! = 1 \cdot 3 \cdot 5 \cdot 7 \cdots (2n-7)\,$ distinct 2-trees.
Of course, many pairs of these 2-trees will be isomorphic as graphs. 

The number of 2-trees up to isomorphism appears 
as the entry A054581 in the Online Encyclopedia of Integer Sequences (OEIS). 
That sequence begins with the counts
$$  1,1, {\bf 2}, 5, 12, 39, 136, 529, 2171, 9368, 41534,\ldots \,\,
{\rm for} \,\,\, n\, =\, 4,5,{\bf 6},7,8,9,10,11,12,13 \ldots . $$
For instance, there are ${\bf 2}$ distinct unlabeled 2-trees for $n={\bf 6}$. Representatives are given~by
\begin{equation}
	\label{eq:twotrees}
	T_1 = \{12,13,23,14,34,15,45\} \quad {\rm and} \quad T_2 = \{12,13,23,24,34,25,35\}. 
\end{equation}
We identify each 2-tree  $T$ with a subset of the set $S$ in (\ref{eq:setS}),
by removing the initial edge $12$. It thus specifies a function (\ref{eq:scattering2}).
For example, the 2-tree $T_1$ in (\ref{eq:twotrees}) is identified with
$T$ in (\ref{eq:firstchoice}).
 
Theorem \ref{thm:vier} has two directions. First,  (\ref{eq:scattering2}) has
only one critical point when $T$ is a 2-tree. Second, every maximal
graph with this property is a 2-tree. We begin with the first direction.

\begin{lemma}  \label{lem:exhibits}
Every 2-tree exhibits minimal kinematics.
\end{lemma}

\begin{proof}
Let $m = 2n-6$, and identify the coordinates on $(\CC^*)^m$ with the pairs in $T$.
We write $X_T$ for the $(n-3)$-dimensional subvariety of $(\CC^*)^m$ that is parametrized
by the $2 \times 2$-minors $p_{ij}$ with $(i,j) \in T$. Thus $X_T$ is a very affine variety,
defined as the  complement of an arrangement of $m$ hyperplanes in $\CC^{n-3}$.
The scattering potential (\ref{eq:scattering2}) is the log-likelihood function for $X_T$.
The number of critical points, also known as the ML degree, equals the signed Euler 
characteristic of $X_T$; see \cite{ABFKST, LG,ST}. Therefore, our claim says that $|\chi(X_T)|  = 1$.

To prove this, we use the multiplicativity of the Euler characteristic.
This states that, for any fibration $f: E \to B$, with  fiber $F$, the following relation holds: $\chi(E)=\chi(F) \cdot \chi(B)$.

We proceed by induction on $n$. The base case is $n=4$,
where $T = \{(1,3),(2,3)\}$. This set represents the unique $2$-tree on $[3]$, which is
 the triangle graph with vertices $1,2,3$.
Here, the very affine variety is the affine line $\CC^1$ with two points removed. In symbols, we have
$$ X_T \,=\, \{ \,(x_1,x_1-1) \in \CC^2 \,: \, x_1 \not= 0,1 \,\}
\, = \, \{\,(p_{13},p_{23}) \in (\CC^*)^2 \, :  p_{13} - p_{23} = 1 \,\} \, \simeq \,
 \mathcal{M}_{0,4} . $$
This punctured curve satisfies $\chi(X_T) = -1$, so the base case of our induction is verified.

We now fix $k \geq 4$, and we
assume that $X_T \subset (\CC^*)^{2k-6}$ has Euler characteristic $\pm 1$ for all
 2-trees $T$ with $k-1$ vertices. Note that ${\rm dim}(X_T) = k-3$.
 Let $T'$ be any 2-tree with $k$ vertices.
 The vertex $k$ is connected to exactly two vertices $i$ and $j$. Suppose $i < j$.
 The associated very affine variety $X_{T'}$ lives in $(\CC^*)^{2k-4}$.
 Note that   ${\rm dim}(X_{T'}) = k-2$.

 We write the coordinates on  $(\CC^*)^{2k-4}$ 
 as $(\,p,\,p_{ik}, p_{jk})$, where $p \in (\CC^*)^{2k-6}$.
 Consider the the map  $\,\pi\,:\,(\CC^*)^{2k-4} \rightarrow  (\CC^*)^{2k-6},
 \,(\,p,\,p_{ik}, p_{jk}) \mapsto p\,$ which deletes the last two coordinates.
 
  Let $T$ be the 2-tree without the vertex $k$ and the two edges $ik$ and $jk$. 
  Then $ij$ is an edge of $T$, so $p_{ij}$ is among the coordinates of $p$.
The restriction of $\pi$ to $X_{T'}$ defines a fibration
$$ \pi \,:\, X_{T'} \rightarrow X_T \, , \,\,
 \,(\,p,\,p_{ik}, p_{jk}) \mapsto p . $$
 Indeed, the fiber over $p \in X_T$ equals
  $$ F \, = \, \pi^{-1} (p) \,\, = \,\, \{ \, (p_{ik}, p_{jk}) \in (\CC^*)^2 \,:\, p_{jk} - p_{ik} \, = \, c \,\}, \quad
\hbox{ where $c = p_{ij} \not=0$.} $$
 This is the punctured line above, i.e.~$F \simeq \mathcal{M}_{0,4}$.
 We know that $\chi(F) = -1$.  Using the induction hypothesis, and the
 multiplicativity of Euler characteristic, we conclude that
$$ \chi(X_{T'}) = \chi(F) \cdot \chi(X_T)  \,= \,
- \chi(X_T) \, = \, \pm 1. $$
This means that the 2-tree $T'$ exhibits  minimal kinematics,
and the lemma is proved.
\end{proof}

From the induction step in the proof above, we also see how to write down
the three-term linear equations that define $X_T$.
We display these linear equations  for the 2-trees with $n=6$.

\begin{example}[$n=6$] \label{ex:VVV}
Consider the two 2-trees in (\ref{eq:twotrees}). 
Each of them defines a very affine threefold of Euler characteristic $-1$.
Explicitly, these two threefolds are given as follows:
$$ \begin{matrix} X_{T_1} & = &  V(\,p_{13} - p_{23}- 1,\, p_{14} - p_{34}- p_{13},\, p_{15} - p_{45}-p_{14}\,) 
& \subset \,\, (\CC^*)^6,
\smallskip \\
 X_{T_2} & = &  V(\,p_{13}-p_{23}-1, \, p_{24}-p_{34}-p_{23} , \, p_{25}-p_{35}-p_{23}\,)
 & \subset \,\, (\CC^*)^6. 
 \end{matrix} $$
\end{example}

We now come to the converse direction, which asserts that
2-trees are the only maximal graphs $T$ satisfying $\chi (X_T) = \pm 1$.
This will follow from known results on graphs and matroids.

\begin{proof}[Proof of Theorem \ref{thm:vier}]
Consider any arrangement of $m+1$
hyperplanes in the real projective space $\PP^d$,
such that the intersection of all hyperplanes is empty.
This data defines a matroid $M$ of rank $d+1$ on $m+1$ elements.
The complement of the hyperplanes is a very affine variety $X$
of dimension $d$. Let $\RR^d$ be the affine space
obtained from $\PP^d$ by removing any one of the hyperplanes.
We are left with an arrangement of $m$ hyperplanes in $\RR^d$.
By \cite[Theorem 1.20]{LG},
 the number of bounded regions
in that affine arrangement is equal to $|\chi(X)|$.

The number of bounded regions described above depends
only on the matroid $M$, and it is known as  the {\em beta-invariant}.
 This is a result due to Zaslavsky \cite{Zas}. The beta invariant can be 
computed by substituting $1$ into the reduced characteristic polynomial of~$M$;
see \cite{Oxl}.

In our situation, we are considering arrangements of hyperplanes
in $\RR^d$ of the special types $\{x_i=0\}$, $\{x_j=1\}$, or $\{x_k=x_l\}$.
These correspond to graphic matroids, and here the
characteristic polynomial is essentially the chromatic polynomial
of the underlying graph. Our problem is this: for which graphs
does this affine hyperplane arrangement have precisely one
bounded region? Or, more generally, which matroids have
beta-invariant equal to one?

The answer to this question was given by Brylawski \cite[Theorem 7.6]{Bry}:
the beta-invariant of a matroid $M$ equals one
if and only if  $M$ is series-parallel.
This means that $M$ is a graphic matroid, where the
graph is series-parallel.
We finally cite the following from Bodirsky et al.~\cite[page 2092]{BGKN}:
{\em A series-parallel graph on $N$ vertices has at most $2N - 3$ edges. 
Those having this number of edges are precisely the 2-trees.}
Setting $N=n-1$,
we can now conclude that the 2-trees $T$ are
the only subsets of $S$ that exhibit minimal kinematics for $\mathcal{M}_{0,n}$.
\end{proof}

\section{Horn matrices}
\label{sec3}

A remarkable theorem due to June Huh \cite{Huh} characterizes 
very affine varieties $X \subset (\CC^*)^m$ that have maximum likelihood degree one. The unique
critical point $\hat p$ of the log-likelihood function on $X$ is given by the Horn uniformization,
due to  Kapranov~\cite{Kap}.
Huh's result was adapted to the setting of algebraic statistics by
Duarte et al.~in \cite{DMS}. For an exposition see also \cite[Section 3]{LG}.
The {\em Horn uniformization} can be written in concise notation as follows:
\begin{equation}
\label{eq:hornuni}
 \hat p \,=\, \lambda \star (H s)^H. 
\end{equation}
Here, $H$ is an integer matrix with $m$ columns and
$\lambda$ is a vector in $\ZZ^m$. The pair $(H,\lambda)$ is
an invariant of the variety $X$, referred to as the {\em Horn pair} in \cite{DMS}.
The coefficients $s_{ij}$ in the log-likelihood function (aka~Mandelstam invariants)
form the column vector $s$ of length $m$, so 
$Hs$ is a vector of linear forms in the coordinates of $s$.
The notation $(Hs)^H$ means that we regard each column of $H$
as an exponent vector, and we form $m$ Laurent monomials in
the linear forms $H s$. Finally, $\star$ denotes the Hadamard product
of two vectors of length $m$.  

The formula  for $\hat p$ given in (\ref{eq:hornuni}) is elegant, but it requires
getting used to. We encourage our readers to work through
Example \ref{ex:n=6}, where
$(Hs)^H$ is shown for two Horn matrices $H$.

We now present the main result of this section, namely the
construction of the {\em Horn matrix} $H = H_T$ for
the very affine variety $X_T$ associated to any 2-tree $T$ on $[n-1]$.
The matrix $H_T$ has $3n-9$ rows and  $m = 2n-6$ columns,
one for each edge of $T$. It is constructed inductively as follows.
If $n=4$ with $T = \{13,23\}$, which represents the triangle graph, then
$$
H_T \,\, = \,\,
\begin{bNiceMatrix}[first-row,first-col]
& 13 & 23 \\
& 1  & 0  \\
& 0  & 1  \\
& -1  & \!-1 
\end{bNiceMatrix}.
$$

Now, for $k \geq 5$, let $T'$ be any 2-tree with $k$ vertices,
where vertex $k$ is connected to $i$ and $j$, and $T = T' \,\backslash \{ik,jk\}$
as in the proof of Lemma \ref{lem:exhibits}. Then the Horn matrix for $T'$ equals
$$
H_{T'}\,\,\, = \quad
\begin{bNiceMatrix}[first-row,first-col]
  &  2k{-}8 \,\, {\rm columns} & ik & jk \,\, \\
3k{-12} \,{\rm rows} & H_T & {\bf h}_{ij} & {\bf h}_{ij} \,\,\,\\
& {\bf 0} & 1  & 0  \,\,\,\\
& {\bf 0} & 0  & 1  \,\,\,\\
& {\bf 0} & -1  & \!-1  \,\,\,
\end{bNiceMatrix},
$$
where ${\bf 0}$ is the zero row vector
and ${\bf h}_{ij}$ is the column of $H_T$ that is indexed by the edge $ij$.

\begin{theorem} \label{thm:hornmain}
Given any 2-tree $T$ on $[n-1]$, the $(3n-9) \times (2n-6)$ matrix $H_T$ constructed above
equals the Horn matrix $H $ for the very affine variety $X_T$.
There exists a sign vector $\lambda \in \{-1,+1\}^{2n-6}$ such that
(\ref{eq:hornuni}) is the unique critical point $\hat p$ of the scattering potential
(\ref{eq:scattering2}).
\end{theorem}

Before proving this theorem, we illustrate the construction of $H_T$ and the statement.

\begin{example}[$n=6$] \label{ex:n=6}
We consider the two 2-trees that are shown in (\ref{eq:twotrees}).
In each case, the Horn matrix has nine rows and six columns.
We find that the two Horn matrices are
$$
H_{T_1} = 
\begin{bNiceMatrix}[first-row,first-col]
& 13 & 23  & 14 & 34 & 15 & 45 \\
& 1  & 0    &  1  &   1 &  1  &  1  \\
& 0  & 1    &  0  &   0 &  0  &  0  \\
& -1  & \!-1 &  \!-1 &   \!-1 &  \!-1  & \!-1  \\
& 0  & 0    &  1  &   0 &  1  &  1  \\
& 0  & 0    &  0  &   1 &  0  &  0  \\
& 0  & 0    &  \!-1  &  \!-1 &  \!-1  & \! -1  \\
& 0  & 0    &  0  &   0 &  1  &  0  \\
& 0  & 0    &  0  &   0 &  0  &  1  \\
& 0  & 0    &  0  &   0 &  \!-1  &  \!-1  \\
\end{bNiceMatrix}
\quad {\rm and} \quad
H_{T_2} = 
\begin{bNiceMatrix}[first-row,first-col]
& 13 & 23  & 24 & 34 & 25 & 35 \\
& 1  & 0    &  0  &   0 &  0  &  0  \\
& 0  & 1    &  1  &   1 &  1  &  1  \\
& -1  & \!-1 & \! -1 & \!  -1 & \! -1  & \! -1  \\
& 0  & 0    &  1 &   0 &  0  &  0  \\
& 0  & 0    &  0  &  1 &  0  &  0  \\
& 0  & 0    & \! -1  & \! -1 &  0  &  0  \\
& 0  & 0    &  0  &   0 &  1  &  0  \\
& 0  & 0    &  0  &   0 &  0  &  1  \\
& 0  & 0    &  0  &   0 &  \! -1  & \! - 1 \\
\end{bNiceMatrix}.
$$
For $H = H_{T_i}$,  the column vector $Hs $ has nine entries,
each a linear form in six $s$-variables. Each column of $H$ specifies 
an alternating product of these linear forms,
and these are the entries of $(Hs)^H$. By adjusting signs
when needed, we obtain the  six coordinates of $\hat p$.

For the second 2-tree $T_2$, the six coordinates of the critical point $\hat p$ are
$$ 
\begin{matrix}
\hat{p}_{13} & = & \frac{s_{13}}{s_{13} + s_{23} + s_{24} + s_{34} + s_{25} +  s_{35}} & & 
\hat{p}_{23} & = & - \frac{s_{23} + s_{24} + s_{25} + s_{34} + s_{35}}{s_{13} + s_{23} + s_{24} + s_{34} + s_{25} +  s_{35}}  
\smallskip \\
\hat{p}_{24} & = & - \frac{(s_{23} + s_{24} + s_{25} + s_{34} + s_{35}) s_{24}}{(s_{13} + s_{23} + s_{24} + s_{34} + s_{25} +  s_{35})(s_{24}+s_{34})}
 & & 
\hat{p}_{34} & = &  \frac{(s_{23} + s_{24} + s_{25} + s_{34} + s_{35}) s_{34}}{(s_{13} + s_{23} + s_{24} + s_{34} + s_{25} +  s_{35})(s_{24}+s_{34})}
\smallskip \\
\hat{p}_{25} & = & - \frac{(s_{23} + s_{24} + s_{25} + s_{34} + s_{35}) s_{25}}{(s_{13} + s_{23} + s_{24} + s_{34} + s_{25} +  s_{35})(s_{25}+s_{35})}
 & & 
\hat{p}_{35} & = &  \frac{(s_{23} + s_{24} + s_{25} + s_{34} + s_{35}) s_{35}}{(s_{13} + s_{23} + s_{24} + s_{34} + s_{25} +  s_{35})(s_{25}+s_{35})}
\end{matrix}
$$
For the first 2-tree $T_1$, the six coordinates of the critical point $\hat p$ are
$$ \!
\begin{matrix}
\hat p_{13} &\! \!= \!& \frac{s_{13} + s_{14}  + s_{34}+ s_{15}  + s_{45}}{s_{13} + s_{23} + s_{14} + s_{34} + s_{15}   + s_{45}}
& \! \! &
\hat p_{23} &\! \!=\! & - \frac{s_{23}}{s_{13} + s_{23} + s_{14} + s_{34} + s_{15}   + s_{45}} \smallskip \\
\hat p_{14} &\!\! \! =\! \!\!& \frac{(s_{13} + s_{14}  + s_{34}+ s_{15}  + s_{45}) (s_{14} + s_{15} + s_{45})}{(s_{13} + s_{23} + s_{14} + s_{34} + s_{15}   + s_{45})(s_{14} + s_{15} + s_{34} + s_{45})} &\! \! &
\hat p_{34} &\!\!=\!&\!\!\! - \frac{ (s_{13} + s_{14} + s_{34}+ s_{15} + s_{45}) s_{34}}
{(s_{13} + s_{23} + s_{14} + s_{34} +  s_{15}  + s_{45}) (s_{14}  + s_{34} + s_{15} +  s_{45}) }
\end{matrix} \vspace{-0.1in}
$$
$$ 
\begin{matrix}
\hat p_{15} & \! = \! &  \phantom{-}
\frac{( s_{13} +s_{14} + s_{34} + s_{15}  + s_{45}) (s_{14} + s_{15} + s_{45}) s_{15}}{(s_{13} + s_{23}+ s_{14}  + s_{34} + s_{15} + s_{45}) (s_{14} +  s_{34} + s_{15}  + s_{45})(s_{15} + s_{45})}
\smallskip
\\
\hat p_{45} & \! = \! &  -\frac{( s_{13} +s_{14} + s_{34} + s_{15}  + s_{45}) (s_{14} + s_{15} + s_{45}) s_{45}}{(s_{13} + s_{23}+ s_{14}  + s_{34} + s_{15} + s_{45}) (s_{14} +  s_{34} + s_{15}  + s_{45})(s_{15} + s_{45})}
\end{matrix}
$$
We note that these $\hat p_{ij}$ satisfy the trinomial equations given for
$X_{T_1}$ resp.~$X_{T_2}$ in Example~\ref{ex:VVV}.
\end{example}

\begin{proof}[Proof of Theorem \ref{thm:hornmain}]
We start by reviewing the Horn uniformization (\ref{eq:hornuni}) in the version
 proved by Huh  \cite{Huh}.
Let $X \subset (\CC^*)^m$ be any very affine variety. The following are~equivalent:
\begin{enumerate}
    \item[(i)] The variety $X$ has maximal likelihood (ML) degree 1;
    \item[(ii)] There exists $\lambda=(\lambda_1, \ldots,\lambda_m) \in (\CC^*)^m$ and a matrix
          $H = (h_{ij})$ in $ \ZZ^{\ell \times m}$ with zero column sums and left kernel $A$,
                      such that the monomial map 
               $$ (\CC^*)^\ell \,\to\, (\CC^*)^m, \,\,
               \,\mathbf{q}=(q_1,\ldots,q_\ell) \,\mapsto \,\bigl(\lambda_1 \underset{i=1}{\overset{\ell}{\prod}}q_i^{h_{i1}},\,
                     \ldots,\,\lambda_m \underset{i=1}{\overset{\ell}{\prod}}q_i^{h_{im}} \bigr)$$ 
                     maps the $A$-discriminantal variety $\Delta_A $ dominantly onto $X$.
 \end{enumerate}
 
Constructing a Horn uniformization proves that the variety $X$ has ML degree $1$.
The corresponding Horn map (\ref{eq:hornuni})
is precisely the unique critical point of the log-likelihood function.
 The following argument thus also serves as an alternative proof of Theorem \ref{thm:vier}.

Fix a 2-tree $T$ on $[n-1]$, with associated variety $X= X_T $ in $ (\CC^*)^m$, where $m = 2n-6$.
Let $H = H_T$ be the $(3n-9) \times (2n-6)$ matrix constructed above.
Its left kernel is given by 
$$ A \quad = \quad \begin{bmatrix}
 1 & 1 & 1 & 0 & 0 & 0 & \cdots & 0 & 0 & 0 \\
 0 & 0 & 0 & 1 & 1 & 1 & \cdots & 0 & 0 & 0 \\
  \vdots & \vdots & \vdots & & & & \ddots & &  \\
  0 & 0 & 0 & 0 & 0 & 0 & \cdots & 1 & 1 & 1 \\
\end{bmatrix}.
$$
This matrix has $n-3$ rows. Its toric variety is the
$(n-4)$-dimensional linear space
$$ X_A \,\,=\,\,\bigl\{ \,(\,t_1,t_1,t_1,\,t_2,t_2,t_2,\,\ldots,\,t_{n-3},t_{n-3},t_{n-3}\,) \,:\,
t_1,t_2,\ldots,t_{n-3} \in (\CC^*)^{3n-9}\,\bigr\}  \quad {\rm in}  \,\, \,\PP^{3n-10}. $$
The $A$-discriminantal variety is the variety projectively dual to $X_A$. Here it is the linear~space
 $$ \Delta_A \,\,\, = \,\,\,
  \bigl\{\,
 q =  (q_1,q_2,\ldots, q_{3n-9}) \ : \ q_{3i+1}+q_{3i+2}+q_{3i+3}=0 
  \,\,\,{\rm for} \,\,\, i = 0,1,\ldots,n-4 \,\bigr\}. $$
  Disregarding $\lambda$ for now,
the monomial map in item (2) above takes $q \in \Delta_A$ to the vector
$$
   p  \,\,= \,\, \bigl(\,p_{13}, p_{23},\,\ldots \,,\,p_{ik} \,,\,\, p_{jk}, \,\ldots \,\, \bigr) 
 \,\,    = \,\, \biggl(\,
    \frac{q_1}{q_3}, \,\frac{q_2}{q_3}, \, \ldots \,,\,
    \,p_{ij} \cdot \frac{q_{3k-8}}{q_{3k-6}}\,,\,\, p_{ij} \cdot \frac{q_{3k-7}}{q_{3k-6}} ,\, \ldots \biggr).
$$
Using the trinomial equations that define $\Delta_A$, we can write this as follows
$$ p \,\, = \,\, \biggl(
\frac{q_1}{-(q_1+q_2)}, \frac{q_2}{-(q_1+q_2)},\,\ldots \,,\,\,
p_{ij} \cdot \frac{q_{3k-8}}{-(q_{3k-8}+q_{3k-7})}\,,\,\,p_{ij} \cdot \frac{q_{3k-7}}{-(q_{3k-8}+q_{3k-7})}, \,
\ldots\, \biggr). $$
The above vector $p$ satisfies the equation
$\,p_{13} + p_{23} + 1 = 0\,$ and it satisfies all subsequent equations
$p_{ik} + p_{jk} + p_{ij} = 0$ that arise from the construction of the 2-tree $T$.

In a final step, we need to adjust the signs of the coordinates in order for $p$
to satisfy the equations $p_{ik} - p_{jk} - p_{ij} = 0$ that cut out $X_T$;
see e.g.~Example \ref{ex:VVV}. This is done by replacing $p$
with the Hadamard product $\hat p = \lambda \star p$ for an appropriate sign vector
$\lambda \in \{-1,+1\}^{2n-6}$. This now gives the desired
birational map from $\Delta_A$ onto $X_T$. That map furnishes the rational formula for 
the unique critical point $\hat p$. To obtain the version
$\,\hat p \, =\, \lambda \star (Hs)^H$ seen in (\ref{eq:hornuni}), we note that
the map $s \mapsto Hs$ parametrizes the linear space $\Delta_A$.
This completes the proof.
\end{proof}

We conclude this section by making the coordinates of $\hat p= \lambda \star (Hs)^H$ more 
explicit. Given any 2-tree $T$ and any edge $ij$ of $T$, we write
$[s_{ij}]$ for the sum of all Mandelstam invariants $s_{lm}$ 
where $lm$ is any descendent of the edge $ij$ in $T$.
Here {\em descendent} refers to the transitive closure of
the parent-child relation in the iterative construction of $T$:
the new edges $ik$ and $jk$ are children of the old edge $ij$.
In this case we call $\{i,j,k\}$ a {\em triangle} of the 2-tree $T$.
This triangle is an {\em ancestral triangle} of an edge $lm$ of $T$
if $lm$ is a descendant of the edge $ik$.
The entries of the vector $H s$ are the linear forms $[s_{ij}]$ 
and their negated sums $-[s_{ik}] - [s_{jk}]$.

\begin{corollary}\label{cor:hatp}
The evaluation of
 the Plücker coordinate $p_{lm}$ at the critical point of $L_T$~equals
  \begin{equation}
  \label{eq:hatp}
   \hat{p}_{lm}\,\,=\,\, \pm \,\underset{}{\prod} \frac{[s_{ik}]}{[s_{ik}]+[s_{jk}]},
  \end{equation}
  where the product runs over all ancestral triangles
 $\{i,j,k\}$ of the edge $lm$.
 \end{corollary}

This is a corollary to Theorem \ref{thm:hornmain}.
The proof is given by inspecting the Horn matrix~$H_T$.
It is instructive to  rewrite the rational functions $\hat p_{ij}$
in Example \ref{ex:n=6} using the notation (\ref{eq:hatp}).

\section{Amplitudes}
\label{sec4}

In this section we take a step towards particle physics.
We define an amplitude $m_T$ for any 2-tree $T$.
This is a rational function in the Mandelstam invariants $s_{ij}$.
When $T$ is planar,
$m_T$ is a degeneration of the biadjoint scalar amplitude $m_n$. We saw this
in Example~\ref{ex:intro6}.
The article \cite{CE} introduced minimal kinematics  
as a means to  study such degenerations.

We shall express the amplitude $m_T$ in terms of the Horn matrix~$H=H_T$.
Given any 2-tree $T$, the entries of the column vector $s$
are the $2n-6$ Mandelstam~invariants~$s_{ij}$. The entries of the
column vector $H s$ are $3n-9$ linear forms in $s$.
Here is our main result:

\begin{theorem} \label{thm:treeamp}
Fix a 2-tree $T$ on $[n-1]$.
The amplitude $m_T$ associated with $T $ equals
\begin{equation} \label{eq:huhformula} m_T \,\, = \,\, \prod_{\{i<j<k\}} 
\frac{[s_{ik}] + [s_{jk}]}{ [s_{ik}] \cdot  [s_{jk}] }. \end{equation}
The product is over all triangles in $T$.
This rational function of degree $3-n$ is
the product of $n-3$ linear forms in $Hs$ divided by the
product of the other $2n-6$ linear forms in $Hs$.
\end{theorem}

In order for this theorem to make sense, 
we  first need the definition of the amplitude $m_T$.
We shall work in the framework
of  beyond-planar MHV amplitudes developed by Arkani-Hamed et al.~in \cite{beyond}.
Our point of departure is the observation that every 2-tree on $[n-1]$ defines an on-shell diagram.
Here we view $T$ as a list of $n-2$ triples $ijk$, starting with $123$ and ending with $12n$.
The last triple $12n$ is special because it uses the vertex $n$.
Moreover, it does not appear in the formula (\ref{eq:huhformula}).
For a concrete example, 
we identify the 2-trees $T_1$ and $T_2$ in (\ref{eq:twotrees}) with the following two
on-shell diagrams, one planar and one non-planar:
\begin{equation}
\label{eq:twotrees2}
 T_1 :  \{123, 134, 145, 126 \} \qquad {\rm and} \qquad
T_2 : \{123,234,235, 126\}. 
\end{equation}

With each triple $ijk$ in $T$ we associate the row
vector $p_{jk} e_i - p_{ik} e_j + p_{ij} e_k$, and we define
$M_T$ to be the $(n-2) \times n$ matrix whose rows
are these vectors for all triples in $T$. For instance,
$$ M_{T_2} \,\,\, =\,\,\,
\begin{bNiceMatrix}[first-row,first-col]
& 1 & 2 & 3 & 4 & 5 & 6 \,\\
 & \,p_{23} & -p_{13} & \phantom{-} p_{12} & 0 & 0 & 0\, \\
 & \,0 & \phantom{-} p_{34} & - p_{24} & \,p_{23}\, & 0 & 0 \,\\
 &\, 0 & \phantom{-}p_{35} & -p_{25} & 0 & \,p_{23} \,& 0 \,\\
 & \,p_{26} & - p_{16} & 0  & 0 & 0 & \,p_{12}\, 
\end{bNiceMatrix}.
$$
Note that the kernel of $M_T$ coincides with the row span of the matrix $X$ in (\ref{eq:M0nmatrix}).
This implies that there exists a polynomial $\Delta(M_T)$ of
degree $n-3$ in the Pl\"ucker coordinates such that
the maximal minor of $M_T$ obtained by deleting columns $i$ and $j$
is equal to $\,\pm \,p_{ij} \cdot \Delta(M_T)$.
 
\begin{lemma} \label{lem:counttriangles}
For any 2-tree $T$, the gcd of the maximal minors of the matrix $M_T$ equals
\begin{equation}
\label{eq:counttriangles}
 \Delta(M_T) \,\,= \,\, \prod_{ij} p_{ij}^{v_T(ij)-1}, 
 \end{equation}
where the product is over all edges of $T$, and
 $v_T(ij)$ is the number of triangles  containing~$ij$.
\end{lemma}

\begin{proof}
The rightmost maximal square submatrix of $M_T$ is lower triangular.
Its determinant equals the product of the 
$p_{ij}$ where $\{i,j,k\}$ runs over all triangles in the 2-tree.
By construction of $T$, the number of occurrences of $p_{ij}$ is $v_T(ij)-1$.
We divide this product by $p_{12}$ to get $\Delta(M_T)$, since the triangle $\{1,2,n\}$ 
has to be disregarded for a 2-tree on $[n-1]$.
\end{proof}

Following \cite[equation (2.15)]{beyond}, we define the {\em integrand} associated to the 2-tree $T$ to be
\begin{equation}
\label{eq:treeintegrand}
\mathcal{I}_T \quad = \quad \frac{ \Delta(M_T)^2 } { \prod_{ijk \in T} p_{ij} p_{ik} p_{jk} }.
\end{equation}
This is a rational function of degree $-n$ in the Pl\"ucker coordinates.
Lemma \ref{lem:counttriangles} now implies:

\begin{corollary} \label{cor:assint}
Given any 2-tree $T$, the associated integrand equals
\begin{equation}
\label{eq:treeintegrand2}
 \mathcal{I}_T \,\,= \,\, \prod_{ij} p_{ij}^{v_T(ij)-2}. 
 \end{equation}
\end{corollary}

The integrands for the 2-trees  $T_1$ and $T_2$ 
from our running example in 
(\ref{eq:twotrees}) and (\ref{eq:twotrees2}) are
$$ \mathcal{I}_{T_1} \,\,= \,\,
  \frac{1}{
 p_{16}
 p_{26}
 p_{23}
 p_{34}
 p_{45}
 p_{15}}
\qquad {\rm and} \qquad
\mathcal{I}_{T_2} \,\,= \,\,
\frac{p_{23}}{
 p_{16} p_{26}
p_{13} 
p_{24} p_{34}
 p_{25} p_{35}}
$$
Note that $\mathcal{I}_{T_1}$ is a Parke-Taylor factor
for $n=6$.
We are now
using the unconventional labeling $1,n,2,3,4,\ldots,n-1$ for the vertices
of the $n$-gon. Any triangulation of this $n$-gon is a 2-tree $T$.
Such 2-trees are called {\em planar}.
To be precise, the triangulation consists of the three edges of the triangle $\{1,n,2\}$
 together with the $2n-6$ edges given by the 2-tree $T$.
 
\begin{corollary}  \label{cor:planar}
For any planar 2-tree $T$,
our integrand equals the
Parke-Taylor factor
\begin{equation}
\label{eq:ParkeTaylor} \quad
\mathcal{I}_T \quad = \quad
 \frac{1}{p_{1n} \, p_{n2} \,p_{23} \,p_{34} \, \cdots\, p_{n-2,n-1} \,p_{n-1,1}}
 \,\, = \,\,  PT(1,n,2,3,\dots, n-1) .
  \end{equation}
\end{corollary}

\begin{proof} This was observed in
\cite[Section~3]{beyond}. It follows directly from
Corollary \ref{cor:assint}.
\end{proof}

We now define the {\em amplitude} associated to a 2-tree $T$ to be the following expression:
\begin{equation}
\label{eq:mT}  m_T \,\, = \,\,- \frac{\mathcal{(I}_T)^2}{{\rm Hess}(L_T)}( \hat p ). 
\end{equation}
Here ${\rm Hess}(L_T)$ is the determinant of the Hessian
   of the log-likelihood function
in (\ref{eq:scattering2}). Numerator and denominator are evaluated at the
critical point $\hat p$, which is given by Theorem~\ref{thm:hornmain}.

\begin{example}[$n=6$] Let $T = T_2$ be the non-planar 2-tree in our running example. Then
$$  - {\rm Hess}(L_T)({\hat p}) 
\,=\,\frac{(s_{13} \!+\! s_{23} \!+\! s_{24} \!+\! s_{25} \!+\! s_{34} \!+\! s_{35})^7 
\,(s_{24} \!+\! s_{34})^3\,(s_{25} \!+\! s_{35})^3}
{s_{13}  \,(s_{23} + s_{24} + s_{25} + s_{34} + s_{35})^5 \, s_{24} \,s_{34} \, s_{25} \,s_{35}}
$$
$$ \qquad \qquad\quad = \quad
\frac{ ([s_{13}]+[s_{23}])^7}{[s_{13}]^1 \, [s_{23}]^5} \cdot
\frac{([s_{24}]+[s_{34}])^3}{[s_{24}]^1 \, [s_{34}]^1} \cdot
\frac{([s_{25}]+[s_{35}])^3}{[s_{25}]^1 \, [s_{35}]^1}.
$$
The numerator $\mathcal{(I}_T)^2(\hat p)$ is the same expression but
with each exponent increased by one. Therefore $m_T$ is
equal to the ratio given by the Horn matrix $H_T$ as promised in
Theorem~\ref{thm:treeamp}.
\end{example}

\begin{proof}[Proof of Theorem~\ref{thm:treeamp}]
Fix a 2-tree $T$ on $[n-1]$. For any edge $ij$ of $T$, we write
$a_{ij}$ for the number of descendants of that edge.
For any triangle $\{i , j < k\}$ of $T$, we set
$b_k = a_{ik} + a_{jk}+1$.

A key combinatorial lemma about 2-trees is that
$a_{ij}+1$ equals the sum of the integers
$2 \bigl(2-v_T(lm) \bigr) $ where $lm$
runs over the set ${\rm dec}_T(ij)$ of  descendants of
the edge $ij$. To be~precise,
\begin{equation}
\label{eq:keyobs} \qquad
a_{ij}  + 1\,\,\, =  \sum_{lm \in {\rm dec}_T(ij)} \!\!\!\!\! 2 (2- v_T(lm))
\,\,\quad \hbox{for all edges $ij \not= 12$.}
\end{equation}
The initial edge $12$ is excluded. 
Recall that $v_T(lm)$ is the number the triangles containing the edge $lm$.
We prove (\ref{eq:keyobs}) by induction on the construction of $T$,
after checking it for $n \leq 5$. Indeed, suppose
a new vertex $n$ enters the 2-tree, with edges $rn$ and $sn$.
Then $v_T(rs)$ increases by $1$ and $v_T(rn) = v_T(sn) = 1$.
Otherwise $v_T$ is unchanged. For all ancestors $ij$ of $rs$,
the right hand side and the left hand side of (\ref{eq:keyobs}) increase by $2$.
For $ij \in \{rn,sn\}$, both sides are $2$.
For all other edges $ij$, the two sides remain unchanged.
Hence (\ref{eq:keyobs}) is proved.

We now evaluate $(\mathcal{I}_T)^2$ at $\hat p$
by plugging (\ref{eq:hatp}) into the square of (\ref{eq:treeintegrand2}).
This gives
$$ (\mathcal{I}_T)^2(\hat p) \,\, = \,\,  \begin{small} \biggl(\, \prod_{lm}
(\hat p_{lm})^{v_T(lm) - 2}\, \biggr)^{\! 2} \,\, = \,\,
\prod_{lm} \prod_{ijk} \biggl(\frac{[s_{ik}] + [s_{jk}]}{[s_{ik}]} \biggr)^{\! 2(2 - v_T(lm))}, \end{small}
$$
where the inner product is over ancestral triangles $ijk$ of $lm$. Switching the products yields
$$ \begin{matrix} (\mathcal{I}_T)^2(\hat p) & = & \prod_{ijk}\left[ 
\prod_{lm \in {\rm dec}_T(ik)}   \bigl(\frac{[s_{ik}] + [s_{jk}]}{[s_{ik}]} \bigr)^{\! 2(2 - v_T(lm))} \cdot
\prod_{lm \in {\rm dec}_T(jk)}   \bigl(\frac{[s_{ik}] + [s_{jk}]}{[s_{jk}]} \bigr)^{\! 2(2 - v_T(lm))}
\right] \\ & = & 
  \prod_{ijk}\left[  \bigl(\frac{[s_{ik}] + [s_{jk}]}{[s_{ik}]}\bigr)^{a_{ik}+1} \cdot
                             \bigl(\frac{[s_{ik}] + [s_{jk}]}{[s_{jk}]}\bigr)^{a_{jk}+1} \right].
\end{matrix}
$$
where the product is over all triangles $ijk$ of $T$.
In conclusion, we have derived the  formula
\begin{equation} 
\label{eq:nicenumerator}
 (\mathcal{I}_T)^2(\hat p) \,\,\, = \,\,\,
\prod_{ijk}
\frac{ (\,[s_{ik}] + [s_{jk}]\,)^{b_k+1}}{ [s_{ik}]^{a_{ik}+1}\, [s_{jk}]^{a_{jk}+1}},
\end{equation}

\smallskip

In order to prove  Theorem~\ref{thm:treeamp}, we must
show that the Hessian at the critical point equals
\begin{equation} 
\label{eq:nicedenominator}
 {\rm Hess}(L_T)(\hat p) \,\, = \,\, -
\prod_{ijk } 
\frac{ (\,[s_{ik}] + [s_{jk}]\,)^{b_k}}{ [s_{ik}]^{a_{ik}}\, [s_{jk}]^{a_{jk}}}.
\end{equation}
The proof is organized by an induction on $k$,
where the Mandelstam invariants $s_{ij}$
are transformed as we deduce the desired formula 
for $k$ from corresponding formula for $k-1$.

As a warm-up, it is instructive to examine the case $k=4$, where the Hessian is 
a $1 \times 1$ matrix. The entry of that matrix
is the second derivative of  (\ref{eq:tiny})
evaluated at (\ref{eq:simpleformula}).  We find
$$ {\rm Hess}(L_T)(\hat p) \,\, = \,\,
\frac{\partial^2 L}{\partial x_1^{\,2}} (\hat p) 
\,\, = \,\, -\frac{s_{13}}{\hat x_1^2} - \frac{s_{23}}{(\hat x_1-1)^2} \,\, = \,\,
- \frac{(s_{13} + s_{23})^3}{s_{13}  s_{23}}. $$
This equation matches (\ref{eq:nicedenominator}), and it
serves as the blueprint for
 the identity in (\ref{eq:Hkk}) below.

Our first step towards (\ref{eq:nicedenominator}) is to get rid of the minus sign.
To this end, we write  $\mathcal{H}$ for the Hessian matrix of
the negated scattering potential $L_T$, evaluated at $\hat p$.
Its entries are
$$
\mathcal{H}_{ii}\,=\,\underset{\ell=1}{\overset{n}{\sum}} \frac{s_{i \ell}}{\hat{p}_{i \ell}^{\,2}}
\qquad {\rm and}  \qquad \mathcal{H}_{ij}\,=\,-\frac{s_{ij}}{\hat{p}_{ij}^{\,2}}
\quad {\rm for} \,\,\,i \not= j.
$$
We shall prove that ${\rm det}(\mathcal{H})$ equals the product
on the right hand side
(\ref{eq:nicedenominator}). This will be done by downward induction.
   Let $k$ be the last vertex, connected to earlier vertices $i,j$.
We display the rows and columns of the Hessian that are indexed by the triangle
$\{i,j,k\}$:
  $$ \mathcal{H} \,\, = \,\,\begin{bmatrix}
    ... & ... & ... & ... & ... & 0 \\
    ... & ...+\frac{s_{ij}}{\hat{p}_{ij}^2}+\frac{s_{ik}}{\hat{p}_{ik}^2} & ... & -\frac{s_{ij}}{\hat{p}_{ij}^2} & ... & -\frac{s_{ik}}{\hat{p}_{ik}^2} \\
    ... & ... & ... & ... & ... & 0 \\
    ... &-\frac{s_{ij}}{\hat{p}_{ij}^2} & ... & ...+\frac{s_{ij}}{\hat{p}_{ij}^2}+\frac{s_{jk}}{\hat{p}_{jk}^2} & ... & -\frac{s_{jk}}{\hat{p}_{jk}^2} \\
    ... & ... & ... & ... & ... & 0 \\
     0 & -\frac{s_{ik}}{\hat{p}_{ik}^2} & 0 & -\frac{s_{jk}}{\hat{p}_{jk}^2} & 0 & \frac{s_{ik}}{\hat{p}_{ik}^2}+\frac{s_{jk}}{\hat{p}_{jk}^2}
\end{bmatrix}.$$ 

Since $k$ is a terminal node in the 2-tree $T$, its two edges satisfy 
a variant of~(\ref{eq:simpleformula}), namely
\begin{equation}
\label{eq:terminal} \hat p_{ik}  \,\,= \,\, \frac{s_{ik}}{s_{ik} \!+\! s_{jk}} \cdot \hat p_{ij} \quad {\rm and} \quad
      \hat p_{jk}  \,\,= \,\, \frac{s_{jk}}{s_{ik} \!+\! s_{jk}} \cdot \hat p_{ij} .
\end{equation}
Using these identities, the lower right entry of the matrix $\mathcal{H}$ can be written as follows:
\begin{equation}
\label{eq:Hkk}
\mathcal{H}_{kk} \,\,\,= \,\,\,\frac{s_{ik}}{\hat{p}_{ik}^2}+\frac{s_{jk}}{\hat{p}_{jk}^2}
\,\,\,=\,\,\,\frac{s_{ik}}{\frac{s_{ik}^2 \hat{p}_{ij}^2}{(s_{ik}+s_{jk})^2}}+\frac{s_{jk}}{\frac{s_{jk}^2 \hat{p}_{ij}^2}{(s_{ik}+s_{jk})^2}}\,\,\,=\,\,\,\frac{(s_{ik}+s_{jk})^3}{s_{ik} \,s_{jk}} \cdot \frac{1}{ \hat{p}_{ij}^2} .
\end{equation}
By factoring out $\mathcal{H}_{kk}$ from the last row, 
the lower right entry becomes $1$.
We add~multiples of the last row to rows $i$ and $j$, so as
to cancel their last entries. The resulting upper left block
with one fewer row and  one fewer column is 
the Hessian matrix of the
 scattering potential  for the  2-tree which is
 obtained from $T$ by removing vertex $k$
and its two incident edges $ik$ and $jk$.
However, in the new matrix, $s_{ij}$ is now replaced
by $s_{ij} + s_{ik} + s_{jk}$. Note that this sum equals $[s_{ij}]$
if $v_T(ij) = 2$.
Proceeding inductively, more and more terms get added, and 
eventually each Mandelstam invariant $s_{lm}$ is replaced by 
the corresponding sum $[s_{lm}]$.

In the end, we find that
the determinant of $\mathcal{H}$ is equal to the product of the quantities
\begin{equation}
\label{eq:Hkk2}
\frac{[s_{ik}]}{\hat{p}_{ik}^2}+\frac{[s_{jk}]}{\hat{p}_{jk}^2}
\,\,\,=\,\,\,\frac{([s_{ik}]+[s_{jk}])^3}{[s_{ik}]\, [s_{jk}]} \cdot \frac{1}{ \hat{p}_{ij}^2} ,
\end{equation}
where $ijk$ ranges over all triangles of the 2-tree $T$.
A combinatorial argument like that presented above shows that
this product is equal to (\ref{eq:nicedenominator}). This completes the proof.
\end{proof}

\begin{remark}
Our proof rests on the fact
 that the Hessian is the product of
the expressions (\ref{eq:Hkk2}). Another way to get this is to directly
triangularize the scattering equations $\nabla L_T = 0$.
Using Corollary \ref{cor:hatp}
and the trinomials defining $X_T$, we see that
$\nabla L_T = 0$ is equivalent to
$$ 
\frac{[s_{ik}]}{p_{ik}} \, + \, \frac{[s_{jk}]}{p_{jk}} \,\, = \,\, 0 
\qquad \hbox{for all triangles $ \,\, ijk \,\, $ of $\,T$.}
$$
This is a triangular system of $n-3$ equations in the
unknowns $x_1,x_2,\ldots,x_{n-3}$. The Jacobian 
of this system is upper triangular, and its determinant
is the product of the expressions (\ref{eq:Hkk2}).
\end{remark}

\section{On-Shell Diagrams and Hypertrees}\label{sec5}

On-shell diagrams were developed by Arkani-Hamed, Bourjaily, Cachazo, Postnikov and Trnka \cite{PositiveGrassAmp}.
They encode rational parts, or leading singularities \cite{LS2008}, occurring in scattering amplitudes for $\mathcal{N}=4$ Super Yang-Mills theory \cite[Section 4.7]{PositiveGrassAmp}.  These are defined using the spinor-helicity formalism on a product of Grassmannians, ${\rm Gr}(2,n) \times {\rm Gr}(2,n)$.  We focus on the special case of MHV on-shell diagrams, where leading singularities are rational functions in the coordinates
$p_{ij}$ ($=\langle ij\rangle$, in physics notation \cite[Section 2.3]{beyond}) of a single 
Grassmannian ${\rm Gr}(2,n)$.
The resulting discontinuities of MHV amplitudes are, for us, CHY integrands.

In \cite[Section 2.2]{beyond}, an identification was proposed between on-shell diagrams and certain
collections of $n-2$ triples in $[n]=\{1,\ldots, n\}$.  Among these on-shell diagrams are the hypertrees of 
Castravet-Tevelev \cite{CT}.  This identification is again noted in \cite[Lemma 9.5]{Tev}. 

In algebraic geometry, hypertrees represent effective divisors on
the moduli space $\mathcal{M}_{0,n}$. 
In physics,  one considers also on-shell diagrams for
higher Grassmannians ${\rm Gr}(k,n)$; cf.~\cite{FGPW2015}.
The analogs to hypertrees are now
 $(n-k)$-element collections of $(k+1)$-sets in $[n]$.
These define effective divisors on the configuration spaces
$X(k,n) = {\rm Gr}(k,n)^o/(\CC^*)^n$.
It would be interesting to examine these through the lens of \cite{CT, Tev}.
In this paper, we stay with $k=2$.

\smallskip

We define a {\em hypertree} to be a collection $T$ of
	$n-2$ triples $\Gamma_1,\ldots, \Gamma_{n-2}$ in $[n]$
	such that 
	\begin{enumerate}
		\item[(a)] each $i\in [n]$ appears in at least two triples, and
		\item[(b)] $\,\big\vert \bigcup_{i\in S}\Gamma_i\big\vert \,\ge\, \vert S\vert +2\,$
		for all non-empty subsets $S \subseteq [n-2]$.
	\end{enumerate}
Hypertrees have the same number of triples
as 2-trees. But 2-trees are not hypertrees because
some $i$ appears in only one triple.  However, if the axiom (a) for hypertrees is dropped, but 
(b) is kept, then one obtains all nonzero leading singularities of MHV amplitudes.
 In particular, 2-trees satisfy (b), and one might view them as
 hypertrees in a weak sense. A hypertree is called
 {\em irreducible} if the inequality in (b) is strict
 for $2 \leq |S| \leq n-3$; see
      \cite[Definition 1.2]{CT}.

\begin{example}[$n=6$] \label{ex:running5}
The 2-trees in (\ref{eq:twotrees2}) are not hypertrees. 
The following is a hypertree:
\begin{equation}
\label{eq:running5}
T \,\,=\,\, \{\,123,345,156 ,246\,\}.
\end{equation}
But it is not a 2-tree. 
The hypertree $T$ 
corresponds  to the octahedral on-shell diagram
in \cite[Figure 1.1]{CEGM2019A}. See also
\cite[Figure 2]{CT} and \cite[Section 5.2]{subspaces}.
This hypertree is irreducible, in the sense defined above, and it
will serve as our running example throughout this section.
\end{example}

Many of the concepts for 2-trees from
previous sections make sense for hypertrees $T$.
We define the $(n-2) \times n$ matrix $M_T$ as in Section \ref{sec4}.
The rows of $M_T$ are the vectors $p_{jk} e_i - p_{ik} e_j + p_{ij} e_k$
for $\{i,j,k\} \in T$. These span the kernel of
the $2 \times n$ matrix $X$ in  (\ref{eq:M0nmatrix}).
The gcd of the maximal minors of $M_T$ is a polynomial
$\Delta(M_T)$ of degree $n-3$ in the Pl\"ucker coordinates $p_{ij}$.
The equation $\Delta(M_T) = 0$ defines a divisor in $\mathcal{M}_{0,n}$,
namely the {\em hypertree divisor}.
If $T$ is a 2-tree  then the hypertree divisor is a
union of Schubert divisors $\{p_{ij} = 0\}$, as seen in Lemma \ref{lem:counttriangles}.
For the geometric application in \cite{CT}, this case is uninteresting.
Instead, Castravet and Tevelev focus on hypertree
divisors that are irreducible; see
 \cite[Theorem~1.5]{CT}.

\begin{example}[Irreducible hypertree]
The hypertree $T$ in (\ref{eq:running5}) is irreducible. Its matrix is
$$ M_T \,\,\, =\,\,\,
\begin{bNiceMatrix}[first-row,first-col]
& 1 & 2 & 3 & 4 & 5 & 6 \,\\
 & \,p_{23} & -p_{13} & \, p_{12} \,& 0 & 0 & 0\, \\
 & 0 & 0 & \,p_{45} \,& -p_{35} & \phantom{-} p_{34} & 0 \\
 &\, p_{56} & 0 & 0 & 0 & -p_{16} & \,p_{15} \,\\
 & 0 & \phantom{-} p_{46} & 0 & -p_{26} & 0 &\, p_{24} \,
\end{bNiceMatrix}.
$$
The hypertree divisor for $T$ is an irreducible surface in the threefold $\mathcal{M}_{0,6}$.
It is defined by
$$ \Delta(M_T) \,\, = \,\, p_{12} p_{35} p_{46} \,-\, p_{13} p_{26} p_{45}. $$
This polynomial is irreducible in 
the coordinate ring of ${\rm Gr}(2,6)$.
See also \cite[Figure 2]{CT}.
The corresponding on-shell diagram
appears in \cite[eqn (2.18)]{beyond}.
Different labelings are used.
\end{example}

For every hypertree $T$, we define the 
{\em CHY integrand} $\,\mathcal{I}_T\,$
by the formula in  (\ref{eq:treeintegrand}). This is a rational function of degree $-n$
in the Pl\"ucker coordinates $p_{ij}$.
The {\em scattering potential} for $T$ is the log-likelihood
function $L_T$ in (\ref{eq:scattering2}) where
the sum is over all pairs $(i,j)$ that
are contained in some triple of $T$.
Thus we set $s_{ij} = 0$ in (\ref{eq:scattering})
for all non-edges $ij$ of $T$.
The hypertree $T$ in (\ref{eq:running5}) has
three non-edges, namely $14$, $25$ and $36$.
Furthermore, we use the same formulas as in (\ref{eq:mT}) to define the
 {\em hypertree amplitude} for $T$. To be precise, we set
\begin{equation}
\label{eq:mT2}  m_T \,\, = \,\, \sum_{\hat p} \frac{\mathcal{(I}_T)^2}{{\rm Hess}(L_T)}( \hat p ),
\end{equation}
where the sum ranges over all critical points $\hat p$ of the scattering potential $L_T$.
This is a rational function of degree $3-n$ in the Mandelstam invariants $s_{ij}$
where $(i,j)$ appears in $T$.

\begin{remark} \label{rmk:MLD0}
Our definition of $L_T$ for hypertrees $T$ differs from that 
for 2-trees in (\ref{eq:scattering2}).
 The difference arises from the restriction to the
basis $S$ in (\ref{eq:setS}) which reflects the gauge fixing in 
(\ref{eq:M0nmatrix}).
In the new definition,  $L_T$ has
no critical point when $T$ is a 2-tree.
For instance, let $n = 4$ and $T = \{123, 124\}$.
This $T$ has one non-edge, namely $34$,
which is not in $S =  \{13,23\}$. The new definition
requires us to set $s_{34} = -s_{13} - s_{23}$ to zero,
so that $s_{23} = - s_{13}$. In this case, the function in
  (\ref{eq:tiny}) becomes $L_T = s_{13}\cdot {\rm log}(x_1/(1-x_1))$,
  which has no critical point.
  Minimal kinematics only arises when $s_{34}$ remains
  an unknown, resulting in the critical point (\ref{eq:simpleformula}).
\end{remark}

\begin{example}[Hypertree amplitude] \label{ex:hyperamp}
We compute the amplitude $m_T$ for the hypertree $T$ in (\ref{eq:running5}).  The generic $n=6$ scattering potential $L$ has six critical points $\hat p$, but the restricted scattering potential $L_T$ has only two. 
The expression (\ref{eq:mT2}) makes sense also for $L$. We have
\begin{eqnarray} \label{eq:sumovertwo}
	m_T & = & \sum_{\hat{p}} \frac{1}{\text{Hess}(L)}\left(\frac{
	(p_{12} p_{35} p_{46} \,-\, p_{13} p_{26} p_{45})^2}{(
	p_{12}p_{13}p_{23})(p_{34}p_{35}p_{45})(p_{15}p_{16}p_{56})(p_{24}p_{26}p_{46})}\right)^{\! 2}(\hat{p})
\end{eqnarray}
This evaluates to a rational function in the Mandelstam invariants.
For generic $s_{ij}$, we find
$$ \begin{matrix}
	m_T  \,\,=\,\,  \frac{1}{s_{16} s_{24} s_{35}}+\frac{1}{s_{16} s_{23} s_{45}}+\frac{1}{s_{13} s_{26} s_{45}}+\frac{1}{s_{15} s_{23} s_{46}}+\frac{1}{s_{12} s_{35} s_{46}}+\frac{1}{s_{15} s_{26} s_{34}}
	 +  \frac{1}{s_{12} s_{34} s_{56}}+\frac{1}{s_{12} s_{35} s_{124}}
	\smallskip \\  \quad
	 +\, \frac{1}{s_{24} s_{35} s_{124}}
	 +\frac{1}{s_{12} s_{56} s_{124}}
	 +\frac{1}{s_{24} s_{56} s_{124}}
	 +\frac{1}{s_{13} s_{24} s_{56}} 
	 +  \frac{1}{s_{15} s_{34} s_{125}}
	 +\frac{1}{s_{12} s_{46} s_{125}}
	 +\frac{1}{s_{15} s_{46} s_{125}}
	 +\frac{1}{s_{13} s_{26} s_{134}}\smallskip \\  \quad \,
	 +\,\frac{1}{s_{26} s_{34} s_{134}}+\frac{1}{s_{12} s_{34} s_{125}}
	+  \frac{1}{s_{34} s_{56} s_{134}}+\frac{1}{s_{15} s_{23} s_{145}}+\frac{1}{s_{15} s_{26} s_{145}}+\frac{1}{s_{23} s_{45} s_{145}}+\frac{1}{s_{26} s_{45} s_{145}}+\frac{1}{s_{13} s_{56} s_{134}}
	\smallskip \\ \quad \,\,
	 + \, \frac{1}{s_{16} s_{35} s_{235}}+\frac{1}{s_{23} s_{46} s_{235}}+\frac{1}{s_{35} s_{46} s_{235}}+\frac{1}{s_{13} s_{24} s_{245}}+\frac{1}{s_{16} s_{24} s_{245}}+\frac{1}{s_{16} s_{23} s_{235}} 
 +  \frac{1}{s_{16} s_{45} s_{245}}+\frac{1}{s_{13} s_{45} s_{245}}.
\end{matrix}
$$
We next impose the constraints $s_{14} = s_{25} = s_{36}=0$
coming from the hypertree $T$. Some poles now become spurious.
  By collecting distinct maximal nonzero residues and then canceling spurious poles, we obtain the Feynman diagram expansion for the hypertree amplitude:
$$
\begin{matrix}
		m_T & = &  \frac{1}{s_{16} s_{24} s_{35}}+\frac{1}{s_{16} s_{23} s_{45}}+\frac{1}{s_{13} s_{26} s_{45}}+\frac{1}{s_{15} s_{23} s_{46}}+\frac{1}{s_{12} s_{35} s_{46}}+\frac{1}{s_{13} s_{24} s_{56}}+\frac{1}{s_{12} s_{34} s_{56}}+\frac{1}{s_{15} s_{26} s_{34}} \medskip \\
		& + & \frac{ \phantom{/} s_{15}\,+\,s_{45 }\phantom{/} }{s_{15} s_{23} s_{26} s_{45}}+\frac{\phantom{/}s_{12}\,+\,s_{24}\phantom{/} }{s_{12} s_{24} s_{35} s_{56}}
		+\frac{\phantom{/}s_{12}\,+\,s_{15}\phantom{/} }{s_{12} s_{15} s_{34} s_{46}} + \frac{\phantom{/}s_{13}\,\,+s_{34}\phantom{/} }{s_{13} s_{26} s_{34} s_{56}}+\frac{\phantom{/}s_{13}\,+\,s_{16}\phantom{/} }{s_{13} s_{16} s_{24} s_{45}}+\frac{ \phantom{/}s_{16}\,+\,s_{46}\phantom{/}}{s_{16} s_{23} s_{35} s_{46}}.
\end{matrix}
$$
This is the rational function (\ref{eq:mT2}), where we sum over the two critical points of $L_T$.
This amplitude has $12$ poles,  $24$ compatible pairs, and $14$ Feynman diagrams, in bijection with the faces of a {\em rhombic dodecahedron}, shown in Figure \ref{fig:rhombdod}.
By comparison, the  amplitude in (\ref{eq:m6}) has $ 9$ poles and 
is a sum over $14$ Feynman diagrams, given combinatorially by the
{\em associahedron}.
 \begin{figure}[h]
	\centering
	\vspace{-0.2in}
	\includegraphics[width=0.456\linewidth]{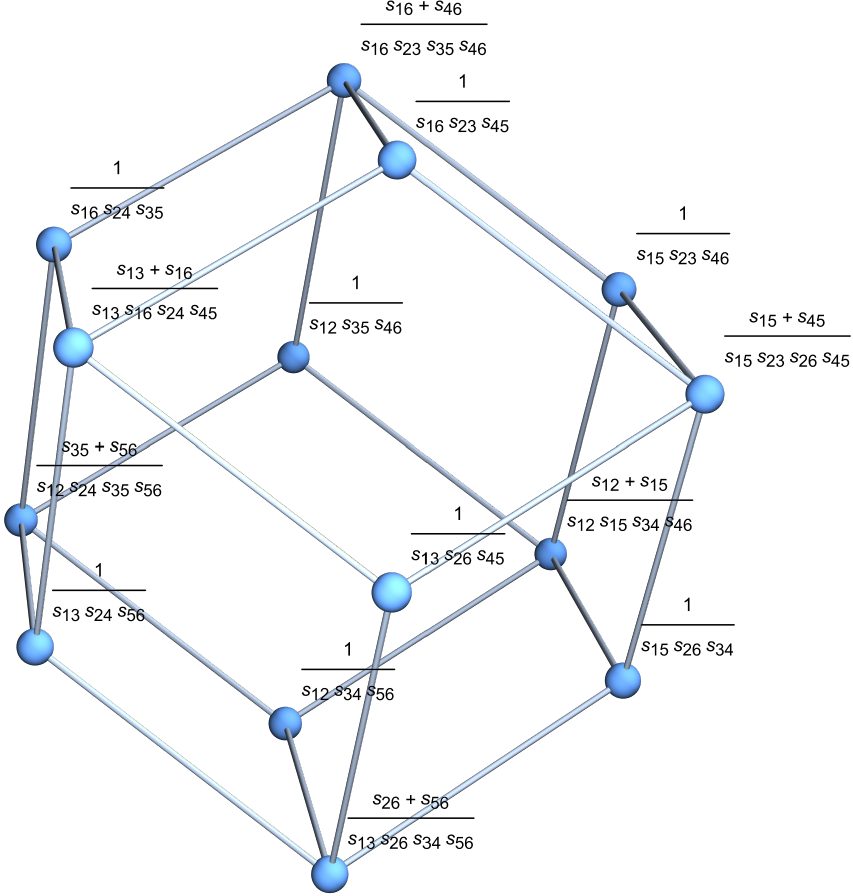}
	\caption{Combinatorics of the octahedral hypertree amplitude $m_T$.}
	\label{fig:rhombdod}
\end{figure}
\end{example}

It was shown in \cite[Section 3.2]{beyond} that, for
 any on-shell diagram $T$, the  CHY integrand $\mathcal{I}_T$
decomposes as a sum of Parke-Taylor factors $PT(\alpha_1,\ldots, \alpha_n)$.
Therefore, $m_T$ is a linear combination of biadjoint amplitudes
$m(\alpha,\beta)$'s as $\alpha,\beta$ range over pairs of cyclic orders on $\lbrack n\rbrack$ 
that are both cyclic shuffles of the triples in $T$.  See \cite{CHYB} for details on this construction.
Such a decomposition with seven terms  is shown in \cite[eqn (3.12)]{beyond}
 for the integrand  $\mathcal{I}_T$  in~(\ref{eq:sumovertwo}).

This example raises several questions for future research.
The first concerns the {\em maximum likelihood degree}
(ML degree) of an arbitrary hypertree $T$. By this we mean the
number of complex critical points of the function $L_T$.
The hypertree $T$ in (\ref{eq:running5}) has ${\rm MLdegree}(T) = 2$.

\begin{question}
Can we find a formula for the ML degree of an arbitrary on-shell diagram, and in particular for an arbitrary hypertree $T$? How is that ML degree related to the geometry of the hypertree divisor
$\{\Delta(M_T) = 0\}$ on the moduli space $\mathcal{M}_{0,n}$?
\end{question}

\begin{remark}
We computed the ML degree for the hypertrees in the Opie-Tevelev database
\url{https://people.math.umass.edu/~tevelev/HT_database/database.html}.
For instance, for $n=9$, the ML degrees range from $8$ to $16$.
For a concrete example consider the hypertree $T = \{123,129,456,789,147,258,367\}$.
Here the ML degree equals $10$, i.e.
the function $L_T$ has $10$ critical points.
By contrast, the general scattering potential $L$ in
(\ref{eq:scattering}) has $720$ critical points,
and its $m(\alpha,\beta)$ expansion 
involves $3185$ Feynman diagrams.
After imposing $s_{ij}=0$ for every non-edge $ij$, we find that
$297$ distinct maximal residues remain in the amplitude $m_T$.  This is still a considerable 
amount of structure to be found from only $10$ critical points.
\end{remark}

Returning to the title of this paper,
we ought to be looking for minimal kinematics.

\begin{question}
For any hypertree $T$, how to best reach ML degree one by
restricting the scattering potential $L_T$ to a subspace of kinematic space?
In particular, can we always reach a Horn uniformization formula  (\ref{eq:hornuni})
for the critical points $\hat p$
  by setting some multiple-particle poles $s_{ij \cdots k}$ to zero?
This would lead to a formula like  (\ref{eq:huhformula})
  for the specialized amplitude $m_T$.
\end{question}

The following computation shows that the answer is ``yes'' for our running example.

\begin{example} \label{ex:yes}
In Example \ref{ex:hyperamp} we set the three-particle pole $s_{234}= s_{23} + s_{34} + s_{24}$ to zero,
 in addition~to $s_{14} = s_{25} = s_{36}=0$.
This results in a dramatic simplification of the amplitude:
$$ m_T\,\, =\,\, \frac{  \left(s_{12}+s_{45}\right) \left(s_{13} + s_{46} \right) \left( s_{26} + s_{35} \right) }{
s_{12} s_{13} s_{26} s_{35} s_{45} s_{46} }. $$
The ML degree is now one. By Huh's Theorem \cite{Huh},
the critical point is given by a Horn pair $(H,\lambda)$.
In short, the subspace  $\{s_{14} = s_{25} = s_{36} = s_{234} = 0\}$ 
exhibits minimal kinematics.
\end{example}

Finally, all of our questions extend naturally from ${\rm Gr}(2,n)$ to ${\rm Gr}(k,n)$.
Using physics acronyms, we seek to extend our amplitudes $m_T$ from
CHY theory \cite{CHY} to CEGM theory~\cite{CEGM2019A}. 
For example, the CEGM potential on the
$4$-dimensional space $X(3,6) = {\rm Gr}(3,6)^o/(\CC^*)^6$~is
$$ L \,\,\, = \sum_{1\le i<j<k\le 6} \!\! \log(p_{ijk}) \cdot \mathfrak{s}_{ijk}. $$
This log-likelihood function is known to have $26$ critical points; see e.g.~\cite[Proposition 5]{ST}.
We now restrict to the kinematic subspace
  $\{\mathfrak{s}_{135}=\mathfrak{s}_{235}=\mathfrak{s}_{246}=\mathfrak{s}_{256}=\mathfrak{s}_{356}=\mathfrak{s}_{245}=0\}$. 
  Then the ML degree drops from $26$ to $1$.
  In short, this subspace exhibits minimal kinematics.

\begin{question}
Can we characterize minimal kinematics for the
configuration space $X(k,n) = {\rm Gr}(k,n)^o/(\CC^*)^n$?
What plays the role that 2-trees have in Theorem \ref{thm:vier}?
Can we determine the
ML degree of on-shell diagrams  for $k \geq 3$?
\ \ Ambitiously, we seek 
{\em an all k and n peek}~\cite{CE}.
\end{question}

                \bigskip
		\bigskip

		\footnotesize
                \noindent {\bf Authors' addresses:}

                \smallskip

                \noindent Nick Early,
                MPI-MiS Leipzig
                \hfill \url{Nick.Early@mis.mpg.de}

                \noindent Ana\"elle Pfister,
                MPI-MiS Leipzig
                \hfill \url{anaelle.pfister@gmail.com}

                \noindent  Bernd Sturmfels, MPI-MiS Leipzig
                \hfill \url{bernd@mis.mpg.de}

\end{document}